\documentclass{amsart}

\usepackage{graphicx}

\newtheorem{theorem}{Theorem}[section]

\theoremstyle{definition}
\newtheorem{definition}[theorem]{Definition}

\newtheorem{corollary}{Corollary}[section]

\theoremstyle{remark}
\newtheorem{remark}[theorem]{Remark}

\numberwithin{equation}{section}

\begin{document}

\title{The theta splitting function}

\author{T. Agama}
\address{Department of Mathematics, African Institute for Mathematical science, Ghana
}
\email{emperortheo1991@gmail.com/emperordagama@yahoo.com}


\subjclass[2000]{Primary 54C40, 14E20; Secondary 46E25, 20C20}

\date{\today}

\dedicatory{}

\keywords{theta splitting; permuted; distribution}

\begin{abstract}
 In this paper we study the Theta splitting function $\Theta(s+1)$, a function defined on the positive integers. We study the distribution of this function for sufficiently large values of the integers.  As an application we show that \begin{align}\sum \limits_{m=0}^{s}\prod \limits_{\substack{0\leq j \leq m\\\sigma:[0,m]\rightarrow [0,m]\\\sigma(j)\neq \sigma(i)}}(s-\sigma(j))\sim s^s\sqrt{s}e^{-s}\sum \limits_{m=1}^{\infty}\frac{e^m}{m^{m+\frac{1}{2}}}.\nonumber
\end{align} and that \begin{align}\sum \limits_{j=0}^{s-1}e^{-\gamma j}\prod \limits_{m=1}^{\infty}\bigg(1+\frac{s-j}{m}\bigg)e^{\frac{-(s-j)}{m}}\sim \frac{e^{-\gamma s}}{\sqrt{2\pi}}\sum \limits_{m=1}^{\infty}\frac{e^m}{m^{m+\frac{1}{2}}}.\nonumber
\end{align}
\end{abstract}

\maketitle

\section{\textbf{Introduction}}
The Gamma function is one of the most useful function in mathematics. It was discovered by the then Swiss mathematician in attempt to generalize the factorial function to the non-integers. It is given by the integral \begin{align}\Gamma(x):=\int \limits_{0}^{\infty}t^{x-1}e^{-t}\nonumber
\end{align}The well-known functional equations $\Gamma(x+1)=x\Gamma(x)$ and $\Gamma(1)=1$ \cite{sebah2002introduction} makes the Gamma function a tractable function, as it allows us to study its algebraic and analytical properties. Indeed when $x$ is an integer, then the Gamma function collapses to the well-known factorial function given by $\Gamma(x+1)=x!$. The factorial function is also a very important, as is manisfested in the analytic constants it contains, such as $\pi$, $e$. It shows up a lot in many combinatorial problems, thus making it an indispensible tool in counting. In the following sequel we introduce the theta splitting function $\Theta(s+1)$, a function defined on the positive integers as
\begin{align}\Theta(s+1):=\begin{cases}1+s\Theta(s) \quad \text{if}\quad s\geq 1\\0 \quad \text{if}\quad s=0\end{cases}.\nonumber
\end{align} This function can be thought of as an analogue of the factorial function. We use this function to answer a certain combinatorial problem. In particular we find the numbers of ways a certain arrangement must be done.
\bigskip

\section{\textbf{The theta splitting function}}
In this section we introduce the Theta splitting function. This function can be considered to playing a role, not exactly, but slightly similar to the factorial function. As it turns out, it becomes very useful in understanding a certain combinatorial problem, which we will illustrate in the sequel.
 
\begin{definition}\label{theta function}
Let $s\geq 1$ be an integer, then we set \begin{align}\Theta(s+1):=\begin{cases}1+s\Theta(s) \quad \text{if}\quad s\geq 1\\0 \quad \text{if}\quad s=0\end{cases}.\nonumber
\end{align}
\end{definition}
\bigskip
The $\Theta$ function is connected to the gamma function by a certain relation, which we shall soon state. This relation will enable us to study the $\Theta$ function by infering from the properties of the gamma function $\Gamma$.
Next we state and prove an asymptotic estimate for the theta splitting function on the integers for sufficiently large values of $s$.

\begin{theorem}\label{integers}
For sufficiently large integers values of $s$, we have \begin{align}\Theta(s+1)\sim s^s\sqrt{s}e^{-s}\sum \limits_{m=1}^{\infty}\frac{e^m}{m^{m+\frac{1}{2}}}.\nonumber
\end{align}
\end{theorem}

\begin{proof}
If $s$ is an integer, then by definition \ref{theta function}, we can write \begin{align}\Theta(s+1)=1+s+s(s-1)+\cdots +s(s-1)\cdots 2.\nonumber
\end{align}It follows that \begin{align}\Theta(s+1)&=\frac{\Gamma(s+1)}{\Gamma(s+1)}+\frac{\Gamma(s+1)}{\Gamma(s)}+\frac{\Gamma(s+1)}{\Gamma(s-1)}+\cdots +\frac{\Gamma(s+1)}{\Gamma(2)}\nonumber \\&=\Gamma(s+1)\sum \limits_{j=0}^{s-1}\frac{1}{\Gamma(s-j+1)}.\nonumber 
\end{align}Using Stirlings formula \cite{sebah2002introduction}, we can write\begin{align}\Gamma(s-j+1)=(s-j)!\sim \sqrt{2\pi (s-j)}e^{-(s-j)}(s-j)^{(s-j)}.\nonumber
\end{align}By plugging this estimate into the sum, we find that \begin{align}\Theta(s+1)\sim s^s\sqrt{s}\sum \limits_{j=0}^{s-1}\frac{1}{e^j(s-j)^{s-j+\frac{1}{2}}}.\nonumber
\end{align} By letting $m=s-j$, we can write \begin{align}\Theta(s+1)&\sim s^s\sqrt{s}e^{-s}\sum \limits_{m=1}^{s}\frac{e^m}{m^{m+\frac{1}{2}}}\nonumber \\&\sim s^s\sqrt{s}e^{-s}\sum \limits_{m=1}^{\infty}\frac{e^m}{m^{m+\frac{1}{2}}}\nonumber
\end{align}and the result follows immediately.
\end{proof}

\begin{remark}
Next we state an immediate consequence of this result. It gives us an estimate for a certain product permuted over the integers.
\end{remark}

\begin{corollary}
For sufficiently large integer values $s$, we have \begin{align}\sum \limits_{m=0}^{s}\prod \limits_{\substack{0\leq j \leq m \\ \sigma:[0,m]\rightarrow [0,m]\\\sigma(j)\neq \sigma(i)}}(s-\sigma(j))\sim s^s\sqrt{s}e^{-s}\sum \limits_{m=1}^{\infty}\frac{e^m}{m^{m+\frac{1}{2}}}.\nonumber
\end{align}
\end{corollary}

\begin{proof}
We observe that by definition\begin{align}\Theta(s+1)=1+\sum \limits_{m=0}^{s}\prod \limits_{\substack{0\leq j \leq m\\ \sigma:[0,m]\rightarrow [0,m]\\\sigma(j)\neq \sigma(i)}}(s-\sigma(j)).\nonumber
\end{align}Using Theorem \ref{integers}, the result follows immediately.
\end{proof}

\begin{theorem}
For sufficiently large values of the integers $s$\begin{align}\sum \limits_{j=0}^{s-1}e^{-\gamma j}\prod \limits_{m=1}^{\infty}\bigg(1+\frac{s-j}{m}\bigg)e^{\frac{-(s-j)}{m}}\sim \frac{e^{-\gamma s}}{\sqrt{2\pi}}\sum \limits_{m=1}^{\infty}\frac{e^m}{m^{m+\frac{1}{2}}},\nonumber
\end{align}where $\gamma$ is the euler-macheroni constant.
\end{theorem}

\begin{proof}
By definition, we can write \begin{align}\Theta(s+1)=\Gamma(s+1) \sum \limits_{j=0}^{s-1}\frac{1}{\Gamma(s+1-j)}.\nonumber
\end{align}Using the identity (See \cite{sebah2002introduction})\begin{align}\frac{1}{\Gamma(s+1-j)}=e^{(s-j)\gamma}\prod \limits_{m=1}^{\infty}\bigg(1+\frac{s-j}{m}\bigg)e^{\frac{-(s-j)}{m}}\nonumber
\end{align}we can write \begin{align}\sum \limits_{j=0}^{s-1}e^{\gamma (s-j)}\prod \limits_{m=1}^{\infty}\bigg(1+\frac{s-j}{m}\bigg)e^{\frac{-(s-j)}{m}}&=\frac{\Theta(s+1)}{\Gamma(s+1)}.\nonumber
\end{align}Using Stirlings formula \cite{sebah2002introduction} and Theorem \ref{integers} the result follows immediately.
\end{proof}
\bigskip

\section{\textbf{Distribution of the theta splitting function}} 
In this section we present a table that gives the distribution of the theta  splitting function $\Theta(s)$ function for the first fifteen values of the integers. This function just like the factorial function grows dramatically quickly. We match the actual values of the Theta splitting function for the first fifteen integers as against approximate values using the asymptotic formula in Theorem \ref{integers}.
 
\begin{table}[ht]
\caption{}\label{eqtable}
\renewcommand\arraystretch{1.5}
\noindent\[
\begin{array}{|c|c|c|c|c|c|c|c|c|c|c|c|}
\hline
Values~of s & 1 & 2 & 3 & 4 & 5 & 6 & 7 & 8 & 9 & 10 & 11\\
\hline
\Theta(s)& 0 & 1 & 3 & 10 & 41 & 206 & 1237 & 8660 & 69281 & 623530 & 6235301\\
\hline 
\end{array}
\]
\end{table}

\begin{table}[ht]
\caption{}\label{eqtable}
\renewcommand\arraystretch{1.5}
\noindent\[
\begin{array}{|c|c|c|c|c|}
\hline
Values~of s & 12 & 13 & 14 & 15 \\
\hline
\Theta(s)& 68588312 & 823059745 & 1.069977669\times 10^{10} & 1.604966503\times 10^{11}\\
\hline 
\end{array}
\]
\end{table}

\begin{table}[ht]
\caption{}\label{eqtable}
\renewcommand\arraystretch{1.5}
\noindent\[
\begin{array}{|c|c|c|c|c|c|c|c|c|c|}
\hline
Values~of s & 1 & 2 & 3 & 4 & 5 & 6 & 7 & 8\\
\hline
s^s\sqrt{s}e^{-s}\sum \limits_{m=1}^{\infty}\frac{e^m}{m^{m+\frac{1}{2}}}& 0 & 1.688 & 3.514 & 10.687 & 43.04 & 216.11 & 1300.256 & 9119.823\\
\hline 
\end{array}
\]
\end{table}

\begin{table}[ht]
\caption{}\label{eqtable}
\renewcommand\arraystretch{1.5}
\noindent\[
\begin{array}{|c|c|c|c|c|c|c|c|c|c|c|c|}
\hline
Values~of s & 9 & 10 & 11 & 12 & 13 \\
\hline
s^s\sqrt{s}e^{-s}\sum \limits_{m=1}^{\infty}\frac{e^m}{m^{m+\frac{1}{2}}}& 73067.075 & 658364.17 & 6589733.73 & 72541956.39 & 871052794.3\\
\hline 
\end{array}
\]
\end{table}

\begin{table}[ht]
\caption{}\label{eqtable}
\renewcommand\arraystretch{1.5}
\noindent\[
\begin{array}{|c|c|c|c|c|c|c|c|c|c|c|c|}
\hline
Values~of s & 14 & 15\\
\hline
s^s\sqrt{s}e^{-s}\sum \limits_{m=1}^{\infty}\frac{e^m}{m^{m+\frac{1}{2}}}& 1.132973304\times 10^{10} & 1.586888658\times 10^{11}\\
\hline 
\end{array}
\]
\end{table}

\section{\textbf{Applications}}
Suppose in a large church auditorium there are $s$ seats on each of the $s$ rows. If $\frac{s(s+1)}{2}$ people are to be assigned to a seat one at a time in the auditorium so that the following rules has to be obeyed\begin{enumerate}
\item [(i)] One particular person must be assigned to the first row

\item [(ii)] Two particular people must be assigned to the second row.

\item [(iii)] More generally, some  $j$ particular people must be assigned to the $j$th row for $1\leq j\leq s$.
\end{enumerate} 
Let $\mathcal{N}(s)$ denotes the number of such possible assignments, then \begin{align}\mathcal{N}(s)\sim s^s\sqrt{s}e^{-s}\sum \limits_{m=1}^{\infty}\frac{e^m}{m^{m+\frac{1}{2}}}.\nonumber
\end{align}
\footnote{
\par
.}%
.




\bibliographystyle{amsplain}

\begin{thebibliography}{10}




\bibitem {sebah2002introduction}Sebah, Pascal and Gourdon, Xavier, \textit{Introduction to the gamma function}, American Journal of Scientific Research, 2002.






\end{thebibliography}

\end{document}